\documentclass[11pt]{article}   	                 
\usepackage{amsmath,amsthm, amssymb,amsfonts,amscd,mathabx}

\usepackage[colorlinks=true,linkcolor=blue,citecolor=red]{hyperref}
\usepackage[pdftex]{graphicx}                 
\usepackage{wrapfig}                               
\usepackage[font=footnotesize]{caption}                     
\usepackage{float}

\numberwithin{equation}{section}
\newtheorem{theorem}{Theorem}[section]
\newtheorem{lemma}[theorem]{Lemma}
\newtheorem{proposition}[theorem]{Proposition}
\newtheorem{definition}[theorem]{Definition}

\newtheorem{remark}[theorem]{Remark}

\newenvironment{Proof}[1][.]%
{\begin{trivlist}\item[]\textbf{Proof#1 }}%
{\qed\end{trivlist}}

 {\begin{trivlist}\item[]\textbf{Acknowledgments }}{\end{trivlist}}

\newcommand{\R}{\mathbb{R}}

\newcommand{\Rmnum}[1]{\uppercase\expandafter{\romannumeral #1\relax}}

\newcommand{\caO}{\mathcal{O}}
\newcommand{\rmO}{\mathrm{O}}

\newcommand{\rmd}{\mathrm{d}}
\newcommand{\rme}{\mathrm{e}}

\renewcommand{\Re}{\mathrm{Re}}

\renewcommand{\leq}{\leqslant}
\renewcommand{\geq}{\geqslant}

\def\eps{\varepsilon}

\newfam\bifam
\font\tenbi=cmmib10 scaled \magstep1 \font\sevenbi=cmmib10 at 11pt
\font\fivebi=cmmib10 at 6pt \textfont\bifam = \tenbi
\scriptfont\bifam = \sevenbi \scriptscriptfont\bifam= \fivebi

\title{Pitchfork bifurcation along a slow parameter ramp: coherent structures in the critical scaling}
\author{
Ryan Goh\thanks{Department of Mathematics and Statistics, Boston University, 665 Commonwealth Ave., Boston,  MA 02215, USA; \texttt{rgoh@bu.edu}.},
 Tasso J. Kaper\thanks{Department of Mathematics and Statistics, Boston University, 665 Commonwealth Ave., Boston,  MA 02215, USA},
  Arnd Scheel\thanks{School of Mathematics, University of Minnesota, 206 Church St. SE, Minneapolis,  MN 55455, USA.}, 
}

\begin{document}



\maketitle

\begin{abstract}\noindent 
We investigate the slow passage through a pitchfork bifurcation in a spatially extended system, when the onset of instability is slowly varying in space. We focus here on the critical parameter scaling, when the instability locus propagates with speed $c\sim \eps^{1/3}$, where $\eps$ is a small parameter that measures the gradient of the parameter ramp. Our results establish how the instability is mediated by a front traveling with the speed of the parameter ramp, and demonstrate scalings for a delay or advance of the instability relative to the bifurcation locus depending on the sign of $c$, that is on the direction of propagation of the parameter ramp through the  pitchfork bifurcation. The results also include a generalization of the classical Hastings-McLeod solution of the Painlev\'e--II equation to Painlev\'e-II equations with a drift term. 
\end{abstract}

 
%
%
%




\section{Introduction}

Directional quenching mechanisms have proven to be useful tools in mediating and controlling the formation of coherent structures in various types of physical systems. Here, some sort of external mechanism travels across the medium, progressively rendering it unstable, and subsequently a selected front or patterned state invades the unstable state. Thus, by controlling the quenching process, one hopes to control the specific pattern that is formed, and to suppress the common defect formation observed when small fluctuations excite a homogeneous unstable state; see \cite{Goh_2023} for a general review. This work is motivated by quenching processes which vary slowly in space. Such quenches have found relevance in fluid dynamics \cite{chomaz99,huerre1990,riecke1986pattern,riecke1987perfect}, biology \cite{HISCOCK2015408}, and more generally in nonlinear systems \cite{haberman79,hunt91,kramer82,kuske,maree96}. When viewed as a spatial dynamics system, slowly-ramped quenches lead to dynamic bifurcations or ``slow passage" problems.

Motivated by these phenomena, we are interested in slow passage through a pitchfork bifurcation in a spatially extended system 
\begin{equation}
 u_t=u_{xx}+cu_x + \mu u - u^3.
\end{equation}
We think of the pitchfork as driven by a slowly varying parameter $\mu=\mu(\eps x)$, with $\mu(\eps x)x<0$, $\,0< \eps \ll 1$, and spatial coupling through diffusion and drift with speed $c$. In the lab frame, this corresponds to a parameter ramp moving with speed $c$. This scenario was recently analyzed in the case of $c$ fixed with $0<c<2$ and in the case of $c=0$, \cite{gksv}. There, it was shown that, for the specific equation with $\mu(\eps x)=-\tanh(\eps x)$,
\begin{equation}\label{e:actanh}
 u_t=u_{xx}+cu_x -\tanh(\eps x) u - u^3,
\end{equation}
and for small enough $\eps$,
the system supports stable fronts $u_*(x)$ with $u_*(x)\to 0$ for $x \to\infty$ and $u_*(x)\to 1$ for $x \to -\infty$. 

Interestingly, for each $0<c<2$ fixed, the front exhibits a spatio-temporal delay of instability:
\begin{align} \label{e:mu_fr}
u_*(x)&\sim 0 \ \ {\rm when} \ \ \mu(\varepsilon x)<c^2/4+ a_1 \varepsilon^{2/3} 
+\rmO(\eps\ln\eps), \nonumber \\
u_*(x)&\sim \sqrt{\mu(\varepsilon x)} \ \  {\rm when} \ \ \mu(\varepsilon x)>c^2/4+ a_1 \eps^{2/3}+\rmO(\eps\ln\eps),
\end{align}
where $a_1 = \Omega_0\left( 1 - {c^4/16}\right)^{2/3}$ and $\Omega_0$ is the smallest positive zero of the following linear combination of Bessel functions of the first kind, $J_{-1/3} (2 z^{3/2}/3) + J_{1/3}(2z^{3/2}/3)$; see Theorem 1.1 of \cite{gksv}. Therefore, in the case of $0<c<2$ fixed, there is an $\rmO(\varepsilon^{-1})$ region in space where the stable front is close to an unstable state, since the trivial state destabilizes at $\mu=0$ but the solution stays near the trivial state until just beyond $\mu=c^2/4$. It turns out that in most of this region, up to $\mu=c^2/4$, this instability is only convective, justifying much of this large delay, yet leaving still an $\rmO(\eps^{-1/3})$ wide region where the front is near an absolutely unstable state. 

On the other hand, in the case $c=0$, there is no such delay. Instead, the front interface exhibits a diffusive spillover of the state with $u>0$ into the stable region $x>0$, as shown in Theorem 1.2 of \cite{gksv}.

This article presents the more delicate asymptotic analysis in the transition regime in parameter space to connect these two regimes and to describe the transition from a spill-over of the instability to a delay of instability. We therefore focus on the situation where $c\sim 0$ is small, yet allow $c<0$ for a more complete understanding. Keys to the analysis of the case $c=0$ are first a good understanding of an inner expansion and second an intricate matching of this inner expansion with the outer solution. The latter is accomplished using geometric desingularization and heteroclinic gluing methods. In this work, we focus on the former part for $c\sim0$.

We are interested in the case where a transition from $u\sim 0$ to $u\sim \sqrt{\mu}$ happens in a region near the origin, where $|\mu|\ll 1$ and hence $\tanh(\eps x)\sim \eps x$. Therefore, we are interested in 
\begin{equation*}
 u_t=u_{xx}+cu_x -\eps x u - u^3.
\end{equation*}
This equation possesses a natural scaling
\begin{equation}\label{e:scale}
 u=\eps^{1/3}\tilde{u}, \qquad x=\eps^{-1/3}\tilde{x},\qquad c=\eps^{1/3}\tilde{c}, \qquad  t=\eps^{-2/3}\tilde{t},
\end{equation}
which leads to
\begin{equation*}
 \tilde u_{\tilde t}=\tilde u_{\tilde x\tilde x}+\tilde c \tilde u_{\tilde x} - \tilde x \tilde u - \tilde u^3.
\end{equation*}
In fact, we showed in \cite{gksv} that the cases $c\gg \eps^{1/3}$ ($\tilde{c} \gg 1$) and $c\ll \eps^{1/3}$ ($\tilde{c}\ll 1$) can be understood as small perturbations of the cases $c=\rmO(1)$ and $c=0$, respectively. For notational simplicity, we drop the tildes throughout Sections \ref{s:asy} - \ref{s:5}, considering the equation
\begin{equation}\label{e:aclin}
 u_t=u_{xx}+cu_x - x u - u^3.
\end{equation}

This work focuses on existence, uniqueness, monotonicity, quantitative asymptotics, and qualitative properties of stationary solutions to \eqref{e:aclin}.
Our first main result is the following:

\begin{theorem}[Existence and Uniqueness of Quenched Fronts]\label{t:1}
For any $c\in\R$, the equation \eqref{e:aclin} has a unique monotonically decreasing stationary solution $u_*(x;c)$ with the properties that
\begin{equation}\label{e:lim}
\lim_{x\to -\infty} \left(u_*(x;c)-\sqrt{-x}\right)=0, \qquad \lim_{x\to\infty}u_*(x;c)=0.
\end{equation}
\end{theorem}

We also derive asymptotics and  qualitative properties for the solutions $u_*(x:c)$, as follows. 
\begin{definition}
A stationary solution $u(x;c)$ of \eqref{e:aclin} is said to be an \emph{admissible solution} if $u$ has limits as in \eqref{e:lim} and $\partial_xu(x;c)<0$ for all $x$.
\end{definition}
Let $u_0(x;c)$ denote an admissble solution.
Define the operator
\begin{equation}\label{e:opL}
 \mathcal{L}_c u:=u_{xx}+cu_x-x u -3u_0^2 u.
\end{equation}
We consider $\mathcal{L}_c$ both as a closed operator on $L^2(\R)$, but also on other function spaces, and simply as applied in a pointwise sense later on. The qualitative information is summarized in the following proposition.

\begin{proposition}[Qualitative Properties of Quenched Fronts]\label{p:1}
The solution $u_*$ of Theorem~\ref{t:1} has the following properties:
\begin{enumerate}
 \item[(A)] \emph{transversality and stability:} $u_*$ is transverse for all $c$ in the sense that $\mathcal{L}_c$ has a bounded inverse on $L^2$; in fact, we have $\Re\,\left(\mathrm{spec}\,\mathcal{L}_c\right)<0$ when $\mathcal{L}_c$ is considered as a closed, densely defined operator on $L^2$;
 \item[(B)] \emph{monotonicity in $c$:} 
 $\partial_c u_*(x;c)<0$ for all $x,c\in\R$;
 \item[(C)] \emph{fast quench:} For sufficiently large $c>0$, the front position, $x_\delta(c)=\sup\{x \,: \, u_*(x;c) > \delta  \}$ satisfies
\begin{equation}
\left|x_\delta(c) - \left(-c^2/4 - \Omega_0 \left(\frac{15}{16}\right)^{2/3}\right)\right| \leq K c^{-1},
\end{equation}
for some constant $K>0$ independent of $c$ and some fixed $\delta >0$ small and independent of $c$.
 
Furthermore, for fixed $\tilde{\delta}>0$, for any continuous function $x(c)$
 \begin{equation}
\lim_{c\to\infty}u_*(x(c);c)= 0 \text{ when } x(c)>-c^2/4+\tilde{\delta},
  \end{equation}
and
\begin{equation}
 \lim_{c\to\infty}(u_*(x(c);c)-\sqrt{-x(c)})= 0 \text{ when }x(c)<-c^2/4-\tilde{\delta};
\end{equation}

\item[(D)] \emph{fast reverse quench:} For fixed $\delta>0$ small, let $x_\delta(c) = \sup\{x \,: \, u_*(x;c) > \delta  \} $. Then 
\begin{equation}
\lim_{c\rightarrow -\infty} x_\delta(c) - \sqrt{-c} = 0.
\end{equation}

\item[(E)] For $c \ge 0$, $u_*(x;c)$ intersects the curve $\sqrt{-x}$ in a unique point on $(-\infty,0)$.

%
\end{enumerate}
\end{proposition}


Figure \ref{f:num} gives numerical results supporting these statements, plotting admissible solutions $u_*(x;c)$ of \eqref{e:aclin} for a range of the scaled $c$ values. For $c>0$, we observe that the front interface locus $x_\delta(c)$ decreases proportional to $-c^2/4$ as $c$ increases. Then, for $c<0$, we observe that as $c$ decreases, the tail of the front spills over into $x>0$, so $x_\delta(c)$ increases, while the region where $u\sim \sqrt{-x}$ recedes to the left.  See Section \ref{s:num} for more description of the numerical methods used to obtain these plots.
\begin{figure}[ht]
 \centering
 \hspace{-0.2in}\includegraphics[width=0.85\textwidth]{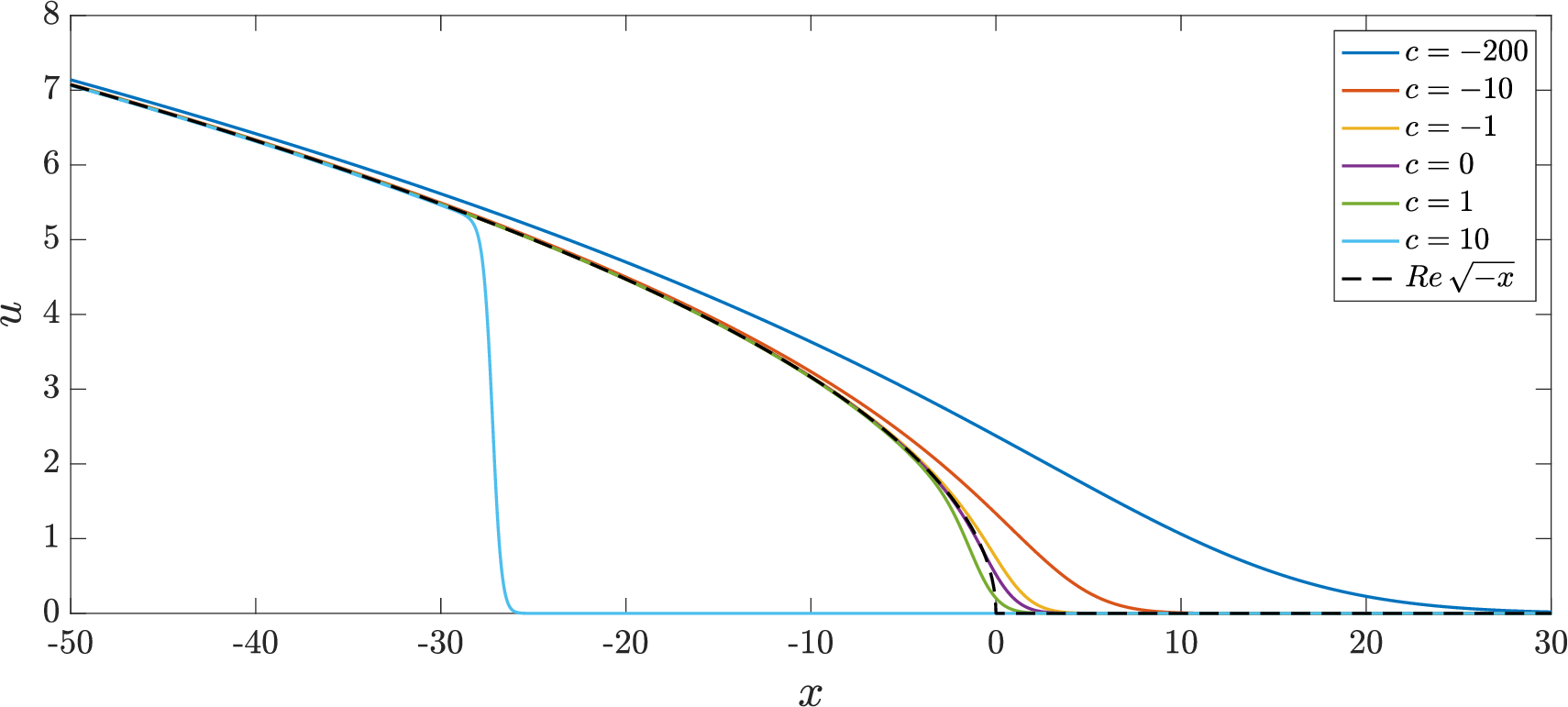}\hspace{-0.15cm}\\
   \hspace{-0.2in}\includegraphics[width=0.33\textwidth]{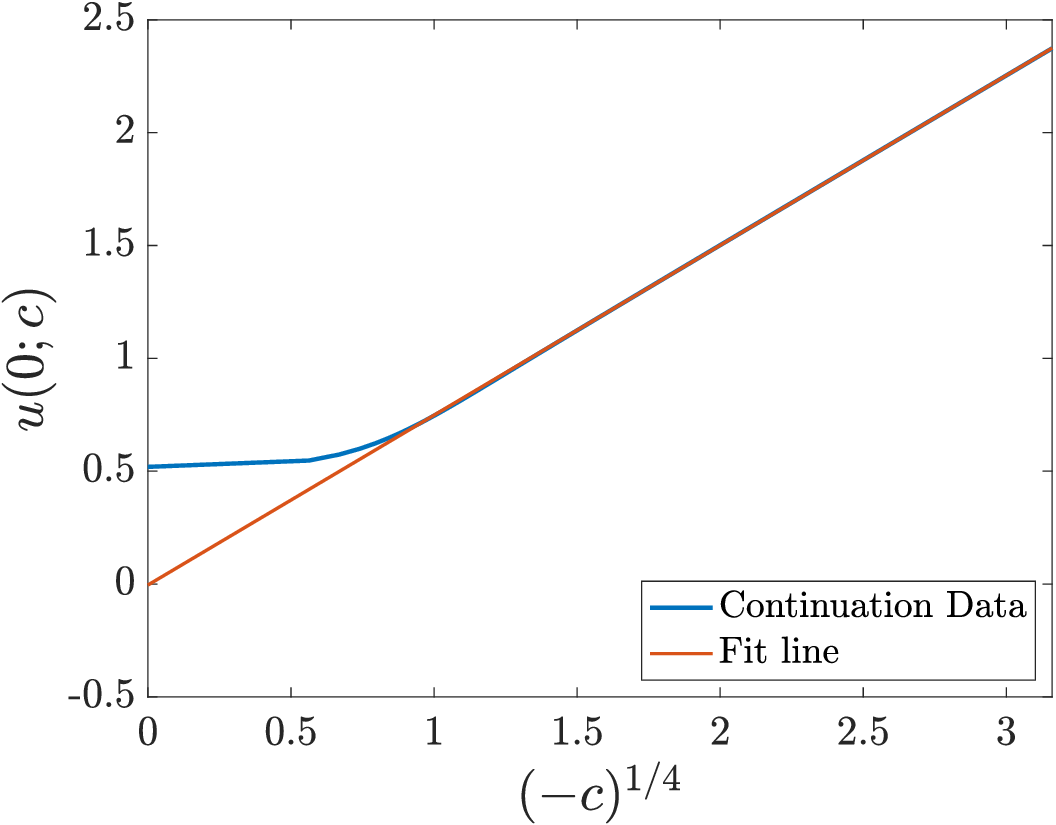}\hspace{-0.0cm}
  \centering \includegraphics[width=0.33\textwidth]{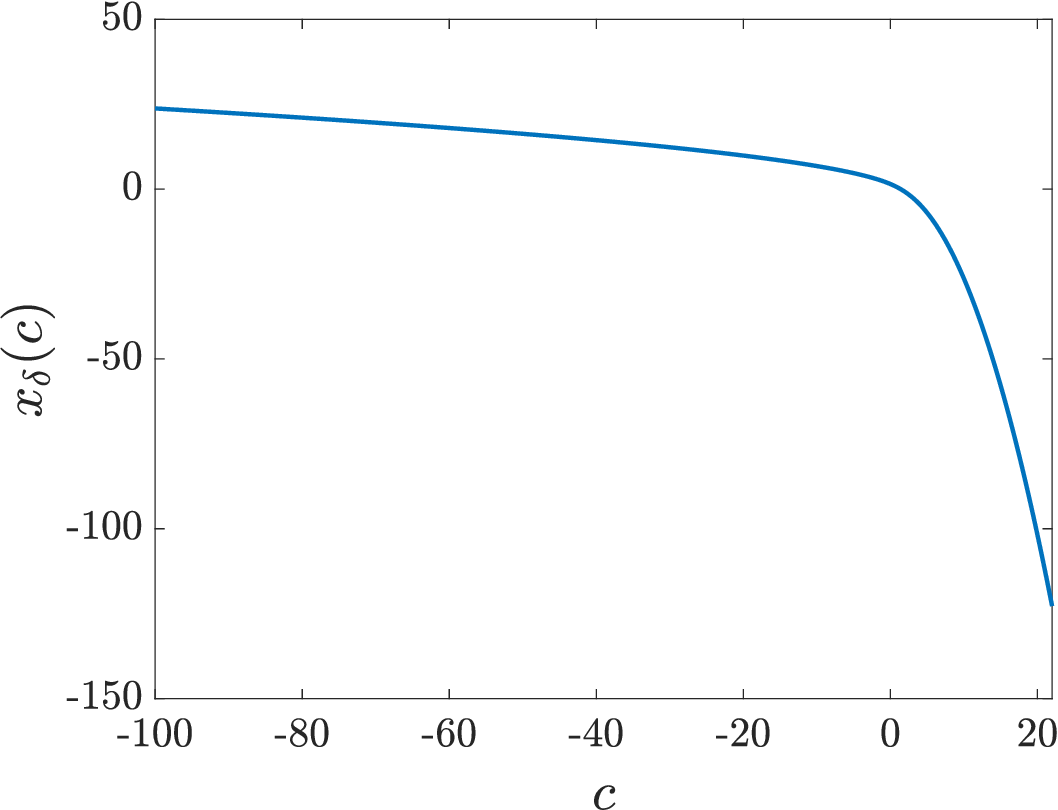}\hspace{-0.1cm}
   \centering \includegraphics[width=0.33\textwidth]{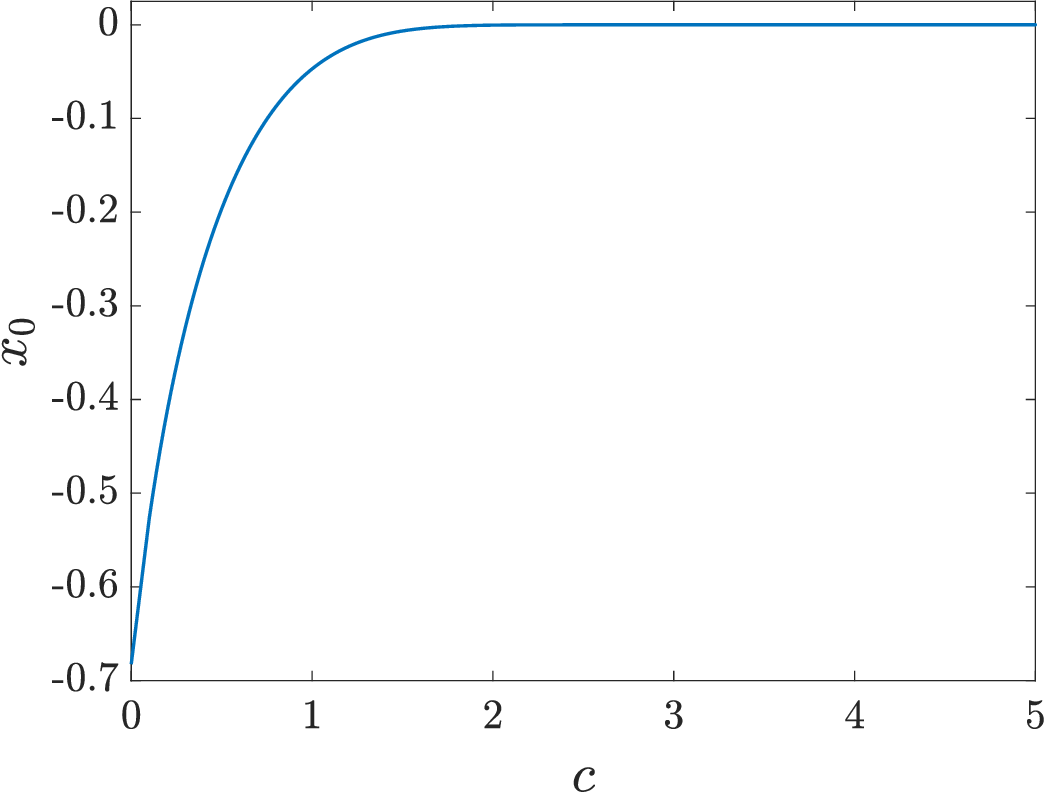}\hspace{-0.2cm}
  \caption{Top: Sample profiles for $c=-200,-10,1,0,1,10$ 
   of admissible solutions $u(x;c)$ of \eqref{e:aclin} computed using AUTO07p. 
This illustrates the transition from fronts with diffusive spill-over (for $c\le 0$) to fronts which exhibit significant delays in the onset of the instability past the pitchfork bifurcation point (for $c>0$). 
  Bottom left: The value $u(0;c)$ plotted as a function of $(-c)^{1/4}$ for $c<0$ (blue) along with linear fit (orange), of data for large $c$ values. Here, the slope of the fit line was found to be 0.7527, within 0.0016 of the predicted $\pi^{-1/4}$; see \eqref{e:u(0)-cnegative}. Bottom center: Value of the invasion point $x_\delta(c)$ defined by $u(x_\delta(c);c)=\delta$ with $\delta = 0.1$. Bottom right: numerically measured crossover point $x_0<0$ from Lemma \ref{l:cross} for a range of positive $c$ values. For $c$ larger, the measured value was within machine precision, and we found that $x u+ u^3<0$ for all grid points with $x<0.$  }
  \label{f:num}
\end{figure}

 We observe that, in the case $c=0$, the equation for stationary solutions of \eqref{e:aclin} is known as the Painlev\'{e}-II equation, and the solution we are interested in was studied in detail by Hastings and McLeod in \cite{hastings80}. The results for $c=0$ in \cite{gksv} rely on this earlier work while adding transversality and stability information. The results here can, in this regard, be viewed as an extension of the work in \cite{hastings80} to the case $c\neq 0$. Our approach does however rely on the stability argument in \cite{gksv} and provides an independent proof also of the existence of the front at $c=0$.

 These results also characterize the inner solution for equilibrium fronts in the slowly ramped Allen-Cahn equation \eqref{e:actanh} with hyperbolic tangent heterogeneity. Such inner heteroclinic solutions can be used as organizing orbits on the singular blow-up sphere and we expect a heteroclinic analysis similar to that of \cite[\S 5]{gksv} to give rigorous existence and asymptotics of the front solutions $u(x;c)$ in \eqref{e:actanh}. See Section \ref{s:num} for numerics supporting this.

\paragraph{Outline of the proof.} The proof consists of the following steps:
\begin{enumerate}
 \item  show that, for all $c$, $\mathcal{L}_c$ is negative for any admissible $u_0(x;c)$;
 \item prove that, for each $c \ll -1$, there exists a unique admissible solution;
  \item prove that, for each $c \gg 1$, there exists a unique admissible solution;
 \item  demonstrate that, if each solution $u_0(x;c)$ in a family of admissible solutions is monotone in $x$, then the solutions are also monotone in $c$, that is, $\partial_c u(x;c)<0$ for all $x$;
 \item prove that
 the set of $c$ such that there exists a solution with $\partial_x u_*(x;c)< 0$ for all $x$ is open;
 \item prove that the set of $c$ such that there exists a solution with $\partial_x u_*(x;c)<0$ for all $x$ is closed, using a priori bounds, a compactness argument, and the second order structure.
\end{enumerate}
Together, these steps will establish both Theorem \ref{t:1} and Proposition \ref{p:1}.

This article is organized as follows. After gathering some results on asymptotics of solutions as $x\to\pm\infty$ in Section \ref{s:asy}, step (i) is carried out in Section \ref{s:2}, the limits in (ii) and (iii) are analyzed in Sections \ref{s:3} and \ref{s:4}, respectively, and steps (iv)-(vi) are performed in Section \ref{s:5} to complete the proofs of Theorem \ref{t:1} and Proposition \ref{p:1}. Further results from numerical simulations are presented in Section~\ref{s:num}, and Section~\ref{s:concl} contains conclusions and discussion.

\section{Asymptotics at \texorpdfstring{$x=\pm\infty$}{Lg} }
\label{s:asy}

Throughout this section, we fix an arbitrary value of the parameter $c\in \mathbb{R}$ and analyze the asymptotics of stationary solutions of \eqref{e:aclin} with the prescribed limits \eqref{e:lim} as $x \to \infty$ and as $x \to -\infty$, beginning with the former.

\begin{lemma}[Stable manifold at $+\infty$]\label{l:asyp}
The set of initial conditions at $x=x_0$ for which stationary solutions of \eqref{e:aclin} satisfy the limit \eqref{e:lim} as $x \to +\infty$ forms a smooth one-dimensional manifold $W^\mathrm{s}_+(x_0;c)$ in the space $(u,u_x)$. The asymptotics of solutions on this stable manifold in the limit $x \to \infty$ are given by
 \begin{equation}\label{e:asyp}
  u(x)=\alpha_+ \exp\left(-\frac{2}{3}\left(x+\frac{c^2}{4}\right)^{3/2}-\frac{c}{2}x\right)x^{-1/4}(1+\rmO(x^{-1})),
 \end{equation}
where the coefficient $\alpha_+$ depends smoothly on the point $(u,u_x)\in W^\mathrm{s}_+(x_0;c)$ and the parameter.
\end{lemma}
\begin{Proof}
The asymptotics of solutions of the steady state equation of \eqref{e:aclin} linearized about $u=0$ are readily obtained by considering $w(x)=\exp(\frac{c}{2}x)u(x)$ which solves a shifted Airy equation $w''-(x+\frac{c^2}{4})w=0$. The asymptotics of the Airy function as $x \to \infty$ are then used to obtain \eqref{e:asyp}. (See, for example formula 9.7.5 in \cite{DLMF} for the expansion of the Airy function.) Solutions to the nonlinear equation are readily obtained by a fixed point argument, which also gives smoothness.
\end{Proof}

\begin{remark}
    We opted to use the  classical approach above to rigorously derive the asymptotics for small $u(x)$ as $x \to \infty$ in the proof of Lemma \ref{l:asyp}, since it is succinct.
There is an  alternate method based on  rewriting the vector field as a third-order autonomous system, effectively conpactifying the independent variable, and using techniques from dynamical systems, including invariant manifold theory and the method of geometric desingularization.
We use this alternate method below in the proofs of Lemmas \ref{l:asym} and \ref{l:asymc}, to rigorously derive the asymptotics for $u(x) \sim \sqrt{-x} $ as $x \to -\infty$, since it brings out the dynamics and geometry of solutions along the unstable manifold
\end{remark}

\bigskip
We now turn to the asymptotics for $x \to -\infty$, considering separately the cases of $c=0$ and $c\ne 0$.

\begin{lemma}[$c=0$: Unstable manifold at $-\infty$]\label{l:asym}
For $c=0$, the set of initial conditions at $x=x_0$ for which the solutions of \eqref{e:aclin} satisfy \eqref{e:lim} in the limit as $x \to -\infty$ forms a smooth one-dimensional manifold $W^\mathrm{u}_-(x_0;c=0)$ in the space $(u,u_x)$.  The asymptotics of solutions on this unstable manifold in the limit $x \to -\infty$ are given by
\begin{equation}\label{e:asym-czero}
  u(x)=(-x)^{1/2}\left(1 - \frac{1}{8(-x)^3} + \rmO(x^{-6}) +\alpha_-\frac{\rme^{-\frac{2\sqrt{2}}{3}(-x)^{3/2}}}{(-x)^{3/4}} +\rmO(\rme^{-\frac{4\sqrt{2}}{3}(-x)^{3/2}})\right),
\end{equation}
where the coefficient $\alpha_-$
depends smoothly on the initial condition in the one-dimensional manifold in the sense that it depends continuously on the point $(u,u_x)\in W^\mathrm{s}_+(x_0;c=0)$.
 \end{lemma}
\begin{remark}
The first terms in \eqref{e:asym-czero}, $\sqrt{-x}\left( 1 - \frac{1}{8} (-x)^{-3} + \mathcal{O}( (-x)^{-6}) \right)$ in the asymptotics for $x \to -\infty$ are precisely the terms in the asymptotic expansion of the Hastings-McLeod solution $u_{HM}(x)$ of the Painlev\'{e}-II equation.
See for example formula (25) in \cite{Cleri_2020}, with $N=3$, where the Painlev\'{e}-II equation is written 
in the form $u''-xu =2u^3$.
\end{remark}

\begin{Proof}
    We perform a series of coordinate changes that exhibit a regular perturbation problem at $x=-\infty$. First, we set $\beta=(-x)^{-1/2}$ to compactify the independent variable, and we find
    \begin{align*}
        u_x&=v,\\
        v_x&=-\beta^{-2}u+u^3,\\
        \beta_x&=\frac{1}{2}\beta^3.
    \end{align*}
    Then, scaling $u=\beta^{-1}\tilde{u}$, $v=\beta^{-2}\tilde{v}$, and $\frac{d}{dx}=\beta^{-1}\frac{d}{dy}$ (which is a natural scaling obtained by using the method of geometric desingularization to desingularize the vector field in the limit $\beta \to 0$), we find 
    \begin{align*}
        \tilde{u}_y&= \tilde{v}+\frac{1}{2}\beta^3\tilde{u},\\
        \tilde{v}_y& =-\tilde{u}+\tilde{u}^3+\beta^3\tilde{v},\\
        \beta_y&=\frac{1}{2}\beta^4,
    \end{align*}
    where $y=-\frac{2}{3}(-x)^{3/2}$. Here, $(\tilde{u},\tilde{v},\beta)=(1,0,0)$ is the equilibrium that corresponds to the desired asymptotic behavior. 
    
    Let $\tilde{u}=1+\bar{u}$ and $\tilde{v} =\bar{v}$, so that the fixed point is at the origin. Then, the system is
    \begin{align} \label{e:ubar-vbar-beta}
        \bar{u}_y&= \bar{v}+\frac{1}{2}\beta^3 + \frac{1}{2}\beta^3 \bar{u},\nonumber \\
        \bar{v}_y& = 2\bar{u}+3\bar{u}^2 + \bar{u}^3+\bar{v} \beta^3,\\
        \beta_y&=\frac{1}{2}\beta^4. \nonumber
        \end{align}
    The linearization at $(0,0,0)$ is hyperbolic in the $(\bar{u},\bar{v})$-plane and possesses a center direction along the $\beta$ axis. Standard invariant manifold theory then gives the existence and smoothness of the center-unstable manifold of the origin. 
    Indeed, the function whose graph is the center-unstable manifold has the following expansion: 
    \begin{equation}
        \label{e:Wc}
\bar{v}= k(\bar{u},\beta)= \sqrt{2}\bar{u} + \frac{1}{\sqrt{2}} \bar{u}^2 - \frac{1}{2} \beta^3
+ \frac{1}{4} \bar{u}\beta^3 - \frac{1}{8} \bar{u}^2 \beta^3 + \frac{1}{4\sqrt{2}} \beta^6 
+\mathcal{O}(\bar{u}^6, \bar{u}^5 \beta, \bar{u}^4 \beta^2, \bar{u}^3\beta^3,
\bar{u}^2 \beta^4, \bar{u}\beta^5, \beta^7).
\end{equation} 
(This is obtained using the invariance condition, and some of the coefficients on the higher order terms vanish.)
Hence, on the center-unstable manifold, the governing equation is
\begin{equation} \label{e:equations-on-Wc}
    \bar{u}_y = \left(\sqrt{2} + \frac{3}{4} \beta^3 + \mathcal{O}(\beta^6)\right) \bar{u} 
    + \left( \frac{1}{\sqrt{2}} - \frac{1}{8} \beta^3 + \mathcal{O}(\beta^6)
    \right) \bar{u}^2 + \mathcal{O} (\bar{u}^3) + 
    \frac{1}{4\sqrt{2}} \beta^6 + \mathcal{O}(\beta^9),
\end{equation}
which is derived by substituting \eqref{e:Wc} for $\bar{v}$
into the first equation of \eqref{e:ubar-vbar-beta}. This is the equation from which we derive the asymptotics.

To derive the asymptotics, we first find the algebraic terms. In particular, the term $\frac{1}{4\sqrt{2}}\beta^6$ in equation \eqref{e:equations-on-Wc} is
the lowest order term that is independent of $\bar{u}$.
Hence, to leading order, the solution $\bar{u}$ is given by 
$ \sqrt{2}\bar{u} + \frac{1}{4\sqrt{2}} \beta^6 =0$, which is the dominant balance.  
That is, $ \bar{u}= -\frac{1}{8}\beta^6= -\frac{1}{18} y^{-2} = - \frac{1}{8} (-x)^{-3}$, to leading order as $\beta\to 0$ and $y ,x  \to -\infty$, respectively. Then, at higher order, one finds terms proportional to higher powers of $(-y)^{-2}$. 

In addition to these algebraic terms involving the inverse powers, the regular perturbation expansion of small solutions $\bar{u}$ for $y \to -\infty$ also contains exponential terms.
Indeed, solving the truncated equation 
$ \bar{u}_y = (\sqrt{2} - \frac{1}{2y} ) \bar{u} + \frac{1}{9\sqrt{2}} y^{-2}$ (where we recall that $\beta= \left(-\frac{3}{2} y\right)^{-1/3}$), one finds
\begin{equation}
    \bar{u}(y) 
    = c_1 e^{2\sqrt{y}} (-y)^{-1/2} 
    - \frac{1}{18} y^{-2}.
\end{equation}

Finally, 
taking into account the higher order nonlinear terms in \eqref{e:equations-on-Wc}, and
translating the solution $\bar{u}(y)$ and its expansion back to the original variables $x$ and $u$ (where we recall $y = -\frac{2}{3} (-x)^{3/2}$ and $u(x)= \sqrt{-x}(1+\bar{u}(x))$), one finds \eqref{e:asym-czero}. 
\end{Proof}

\medskip
A similar, but less degenerate result holds for $c\ne 0$, and we include an outline of the proof.

\begin{lemma}[$c \ne 0$: Unstable manifold at $-\infty$]\label{l:asymc}
For $c\ne 0$, the set of initial conditions at $x=x_0$ through which the solutions of \eqref{e:aclin} satisfy the limit \eqref{e:lim} as $ x \to -\infty$ forms a smooth one-dimensional manifold $W^\mathrm{u}_-(x_0;c)$ in the space $(u,u_x)$.  The asymptotics of solutions on this unstable manifold are given by
 \begin{equation}\label{e:asym-cnotzero}
 \begin{split} u(x)=&\\(-x)^{1/2}&\left(1+ \frac{c}{2\sqrt{2}x^{2}} +\rmO(x^{-3}) +\alpha_-^c {\rm exp} \left( -\frac{2\sqrt{2}}{3}(-x)^{3/2} -\frac{cx}{2} -\frac{c^2}{4\sqrt{2}} (-x)^{1/2}\right) +\rmO(\rme^{-\frac{4\sqrt{2}}{3}(-x)^{3/2}}) \right),
 \end{split}
\end{equation}
where $\alpha_-^c$ depends smoothly on the point $(u,v)\in W^\mathrm{s}_+(x_0;c)$ and the parameter.
 \end{lemma}

 \begin{Proof}
 We use the same coordinates as in the proof of Lemma \ref{l:asym}; however, we observe that one does not need to go to so high an order in the expansion as we did in the case $c=0$ in the previous lemma,
since the system with $c\ne 0$ is less degenerate.
Let $\tilde{u}=1+\bar{u}$ and $\tilde{v} =\bar{v}$, where we recall
$u=\beta^{-1}\tilde{u}$, $v=\beta^{-2}\tilde{v}$,  
$\beta=(-x)^{-1/2}$, 
and $\frac{d}{dx}=\beta^{-1}\frac{d}{dy}$.
With $c\ne 0$, the equation for stationary solutions of \eqref{e:aclin} is equivalent to the system
    \begin{align} \label{e:ubar-vbar-beta-c}
        \bar{u}_y&= \bar{v}+\frac{1}{2}\beta^3 + \frac{1}{2}\beta^3 \bar{u},\nonumber \\
        \bar{v}_y& = -c \beta \bar{v} + 2\bar{u}+3\bar{u}^2 + \bar{u}^3+\bar{v} \beta^3,\\
        \beta_y&=\frac{1}{2}\beta^4. \nonumber
        \end{align}
The linearization at $(0,0,0)$ is hyperbolic in the $(\bar{u},\bar{v})$-plane and possesses a center manifold in the direction of $\beta$. The function whose graph is the center-unstable manifold has the following expansion: 
    \begin{equation}
        \label{e:Wc-c}
\bar{v}= k_c(\bar{u},\beta)= \sqrt{2}\bar{u} + \frac{1}{\sqrt{2}} \bar{u}^2
-\frac{c}{2} \bar{u} \beta 
+\frac{c}{12} \bar{u}^2 \beta+\frac{c^2}{8\sqrt{2}} \bar{u} \beta^2
-\frac{1}{2} \beta^3
-\frac{c}{24} \bar{u}^3 \beta-\frac{7\sqrt{2}c}{288} \bar{u}^2 \beta^2
+\frac{1}{4}\bar{u}\beta^3
+\frac{c}{2\sqrt{2}} \beta^4
+\mathcal{O}(5).
\end{equation} 
Hence, on the center-unstable manifold, the governing equation is
\begin{equation} \label{e:equations-on-Wc-c}
    \bar{u}_y = \left(\sqrt{2} -\frac{c}{2}\beta  + \frac{c^2}{8\sqrt{2}} \beta^2 + \mathcal{O}(\beta^3)\right) \bar{u} 
    + \left( \frac{1}{\sqrt{2}}
    +\frac{c}{12}\beta + \mathcal{O}(\beta^2)
    \right) \bar{u}^2 
    +\frac{c}{2\sqrt{2}} \beta^4 + \mathcal{O}(\bar{u}^3, \beta^5).
\end{equation}
The asymptotic expansion of $\bar{u}(y)$ as $y \to -\infty$ consists of algebraically and exponentially decaying terms, just as that for $c=0$.
The algebraically decaying terms are
$$
\frac{c}{2\sqrt{2}} \left( \frac{-3y}{2}\right)^{-4/3} + \mathcal{O}\left( \left(\frac{-3y}{2}\right)^{-8/3}\right),
$$
by balancing the linear term with the inhomogeneous term.
The exponentially decaying terms are
$$
{\rm exp}\left[ 
\sqrt{2} y + \frac{c}{2} \left( \frac{-3y}{2}\right)^{2/3} -\frac{c}{4\sqrt{2}} \left( \frac{-3y}{2} \right)^{1/3} + \mathcal{O}(1) 
\right].$$
Finally, one translates this back to the original variables $x$ and $u$ to complete the proof of the lemma.
 \end{Proof}

\section{The linearization at monotone solutions}\label{s:2}

In this section, we carry out step (i) in the proof, establishing that the operator $\mathcal{L}_c$, which is obtained by linearizing about a given admissible solution (recall definition \eqref{e:opL}), and which is densely defined on $L^2(\R)$,  is bounded invertible, and in fact has spectrum with negative real part. Therefore, define 
\[
L^2_1(\R)=\{u\in L^2_\mathrm{loc}\,|\,u(x)(1+|x|)\in L^2\}, \qquad \|u\|_{L^2_1}:=\|u(\cdot)(1+|\cdot|)\|_{L^2}.
\]
\begin{lemma}\label{p:lc}
The operator $\mathcal{L}_c$, considered on $L^2(\mathbb{R})$ with domain $H^2(\mathbb{R})\cap L^2_1(\R)$ is closed, densely defined, and bounded invertible. Moreover, its spectrum is discrete and strictly negative. 
\end{lemma}
\begin{proof} As a first step, we consider the  operator $\widetilde{\mathcal{L}}_c u :=\left( e^{c x/2}\mathcal{L}_c e^{-c x/2}\right) u = u_{xx} - V(x) u $, where the potential is defined as $V(x):= x + c^2/4 + 3u_0^2$, and observe the asymptotic behavior 
\[
V(x) \sim -2x + c^2/4 \quad\mathrm{for } \ \ x\rightarrow-\infty,\qquad V(x)\sim x + c^2/4\quad\mathrm{for } \ \ x\rightarrow+\infty,
\]
induced by  $u_0 \sim \sqrt{-x}$, $x\to -\infty$, and  $u_0 \sim 0$, $x\to +\infty$, respectively. Standard results on Schr\"odinger operators then imply that $\widetilde{\mathcal{L}}_c$ is self-adjoint on $L^2(\R)$ with domain $H^2(\mathbb{R})\cap L^2_1(\R)$; see for instance \cite[\S8.6]{hislopsigal}. Moreover, since  $V(x)$ is positive outside of a compact neighborhood of the origin, and $V(x)\rightarrow +\infty$ as $|x|\rightarrow+\infty$, standard results on Schr\"odinger operators (see for example \cite[Thm. XIII.47]{ReedSimonIV}) then give that $\widetilde{\mathcal{L}_c}$ has no continuous spectrum, so that the spectrum consists only of discrete spectrum $\{\lambda_j\}_0^\infty$, which satisfies $\lambda_0\geq\lambda_1\geq \lambda_2 \geq \cdots$ 
and $\lim_{j\to \infty} \lambda_j = -\infty$. Furthermore, possibly after shifting the operator by a finite constant, these results also give that the  ground state eigenfunction corresponding to the eigenvalue $\lambda_0$ is strictly positive, {\it i.e.,} $\phi_0>0$. We claim that the spectrum is in fact strictly negative. To see this, we argue by contradiction. Assume $\lambda_0\geq0$, and differentiate the steady-state equation \eqref{e:aclin} in $x$ to obtain $\mathcal{L}_c \partial_x u_0 = u_0$, and hence $\widetilde{\mathcal{L}}_c e^{cx/2}\partial_x u_0 = e^{cx/2} u_0$. Then, we calculate
$$
\lambda_0 \langle \phi_0, e^{cx/2}\partial_x u_0\rangle_{L^2} = 
\langle \phi_0, \widetilde{\mathcal{L}}_c e^{cx/2} \partial_x u_0\rangle_{L^2} = 
\langle \phi_0, e^{cx/2} u_0\rangle_{L^2} >0.
$$
This is a contradiction, since $\phi_0>0, u_0>0,$ and $\partial_x u_0<0$ by the properties of an admissible solution. This demonstrates that $\lambda_0<0$,
as claimed. 

It remains to conclude the desired properties of the unconjugated operator $\mathcal{L}_c$ from the properties of $\widetilde{\mathcal{L}}_c$. Note that we cannot argue simply that the two operators are conjugate since we wish to establish properties of $\mathcal{L}_c$ on $L^2$ rather than the exponentially weighted space induced by the conjugation. 

First note that the embedding $H^2\cap L^2_1\to H^1$ is compact. This follows readily from the fact that the embedding is the norm limit of truncated embeddings, where $u$ is cut off at $|x|=R$ by a smooth cutoff function. 
Convergence of these truncated embeddings in turn is seen readily by estimating $u$ outside of a small ball using $\|u\|_{H^1}\lesssim \eps \|u\|_{H^2}+C(\eps)\|u\|_{L^2}$, where the latter term is small due to control of decay for $u$ in $L^2_1$.  
As a consequence we can view $\mathcal{L}_c$ as a relatively compact perturbation of  $\widetilde{\mathcal{L}}_c$ which is thereby closed \cite[\S IV.1.3.1.11]{kato} with the same domain of definition. 
In particular, the spectrum of  $\mathcal{L}_c$ consists entirely of point spectrum. 
Inspecting the eigenvalue problem, one immediately finds 
 Gaussian decay for any potential eigenfunction, so that spectra of  $\mathcal{L}_c$ and  $\widetilde{\mathcal{L}}_c$  in fact agree. 
\end{proof}

\medskip
\begin{remark}
    The analysis relies on the more generally valid fact that the eigenfunction to the ground state of the Schr\"odinger operator is sign definite, and on the monotonicity of the parameter ramp. 
\end{remark}

\section{Existence and qualitative properties for 
\texorpdfstring{$c\ll -1$}{Lg}}
\label{s:3}

In this section, we analyze stationary solutions of \eqref{e:aclin} in the limit $c \ll -1$.
We establish the existence, uniqueness, monotonicity, and asymptotics of these solutions. This constitutes step (ii) in the proof.
\begin{lemma}\label{l:cllm1}
For each $c \ll -1$, there exists a unique admissible solution $u(x;c)$ of \eqref{e:aclin}, {\it i.e.,} a unique, monotone decreasing, stationary solution ,which has the asymptotics
$u \to \sqrt{-x}$ as $x \to -\infty$ and $u \to 0$ as $x \to \infty$.
\end{lemma}

\begin{Proof}
For stationary solutions of \eqref{e:aclin} in the regime $c \ll -1$, we scale $y=-c x, u=-c\tilde{u}, \eps=-1/c^3$ and find
\begin{equation}
    \label{e:cbignegative}
 \tilde{u}_{yy}-\tilde{u}_y-\eps y \tilde{u}-\tilde{u}^3=0.
\end{equation}
With $\xi=\eps y$, so that $\xi$ is a slowly varying dependent variable, the governing equation is equivalent to the autonomous system
\begin{align}
 \tilde{u}_y&=\tilde{v} \nonumber \\
 \tilde{v}_y&=\tilde{v}+\xi\tilde{u}+\tilde{u}^3\\
 \xi_y&=\eps. \nonumber
\end{align}
For the frozen system ($\eps=0$), we find for each $\xi>0$ a unique equilibrium at the origin $(0,0)$ in the $(\tilde{u},\tilde{v})$ plane, and it is a saddle. The union of these saddles over all $\xi > \delta$, for some small $\delta >0$, is a normally  hyperbolic invariant manifold, and it persists for sufficiently  small $0<\eps \ll 1$ by Fenichel theory \cite{fenichel} as a slow invariant manifold,  along which $\xi$ slowly increases. Moreover, one can track solutions along this slow manifold backward into a neighborhood of $\xi=0$. This slow manifold, along with the union of its strong stable fibers, constitutes the stable manifold described in Lemma \ref{e:asyp}.

Then, for each $\xi<0$, there is a unique positive equilibrium given by $\tilde{u}=\sqrt{-\xi}$, which is also a saddle. The union of these saddles over all $\xi<-\delta$ is a normally hyperbolic invariant manifold. This also persists for $0<\eps \ll 1$ by Fenichel theory, and one can track solutions along this slow manifold into the neighborhood of the origin. The union of this slow manifold and its strong unstable fibers constitutes the unstable manifold described in Lemmas \ref{l:asym} and \ref{l:asymc}.

In a vicinity of the origin, there is a two-dimensional center manifold which is the graph of $\tilde{v} = h_c(\tilde{u},\xi,\eps) = -\xi \tilde{u} - \tilde{u}^3 + {\rm h.o.t.}$. The dynamics on this center manifold are given to leading order by
\begin{equation} \label{e:equationonWc}
 \tilde{u}_y+\xi\tilde{u}+\tilde{u}^3=0,\qquad \xi_y=\eps.
\end{equation}
This is a one-fast one-slow system, and it exhibits slow passage through a pitchfork bifurcation. Analysis of the slow passage through this pitchfork, following the general approach of \cite{Krupa_2001}, shows that there is a unique solution with asymptotics $\tilde{u}\sim \sqrt{-\xi}$ for $y\to -\infty$ and $\tilde{u}\to 0$ for $y\to\infty$. This unique solution then lies in the transverse intersection of the unstable manifold at $y = -\infty$ and stable manifold at $y = +\infty$.

Additional information about the leading order asymptotics of this unique solution is obtained as follows. Recall that $\xi=\eps y$, so that system \eqref{e:equationonWc} may be written as a Bernoulli equation with power three: $\eps \tilde{u}_\xi +\xi \tilde{u} +\tilde{u}^3=0.$ Then, set $w=\tilde{u}^{-2}$. The new dependent variable $w$ satisfies the nonautonomous linear equation $\eps w_\xi = 2 \xi w +2$, and the solution of interest is $ w(\xi)= \sqrt{ \frac{\pi}{\eps}}e^{\xi^2/\eps}\left( {\rm erf} \left( \frac{ \xi}{\sqrt{\eps}} \right) + 1 \right)$, where ${\rm erf}(x) = \frac{2}{\sqrt{\pi}} \int_0^x e^{-t^2} dt$. 
(Here, we integrated from $\xi_0$ to $\xi$ and then took the limit as $\xi_0\to -\infty$, noting that ${\rm erf}(\xi_0/\sqrt{\eps}) \to -1$ in this limit and that the homogeneous term vanishes due to the strong decay of $e^{-\xi_0^2/\eps}$.)

Translating this back from $w$ to $\tilde{u}$, one finds
\begin{equation} \label{e:utilde-sec4}
\tilde{u}(\xi;\eps)= \frac{ \eps^{1/4} e^{-\frac{\xi^2}{2\eps}}}{\pi^{1/4} \left( {\rm erf} \left( \frac{\xi}{\sqrt{\eps}}\right) +1
\right)^{1/2}}.
\end{equation}
This solution decreases monotonically in $\xi$.
It has the following asymptotics:
\begin{equation} \label{e:utilde-asymptotics}
\begin{split}
\tilde{u} (\xi;\eps) &{\sim}  0 \ \ \ \ {\rm as} \ \ \xi \to \infty, \\
\tilde{u}(\xi;\eps) &{\sim} \sqrt{-\xi}\left(
1+ \Sigma_{n=1}^{\infty} (-1)^n \frac{ (2n-1)!! \eps^n}{2^n \xi^{2n}}
\right)^{-1/2}
\  \ \ {\rm as } \ \  \xi \to -\infty.
\end{split}
\end{equation}
For the limit $\xi \to \infty$, we used ${\rm erf}(\xi/\sqrt{\eps}) \to 1$; and, for the limit $\xi \to -\infty$, we used ${\rm erf}(-x)=-{\rm erf}(x)$, ${\rm erf}(x) = 1-{\rm erfc}(x)$, and 
${\rm erfc}(x) = \frac{e^{-x^2}}{\sqrt{\pi}x} 
\left[ 
1 + \Sigma_{n=1}^\infty (-1)^n \frac{(2n-1)!!}{2^n x^{2n}} \right]$
as $x \to \infty$.
Hence, for each $c \ll -1$, there exists a monotonically decreasing solution with the asymptotics \eqref{e:lim}, making it an admissible solution.

The touch-down point of the solution, that is, the point where the solution enters a small fixed neighborhood of $\tilde{u}=0$ is readily obtained from scaling,
\[
 \xi=-\eps^{1/2}, \qquad |y|\sim \eps^{-1/2}\sim |c|^{3/2}, \qquad x\sim |c|^{1/2},
\]
where we recall that $\xi=\eps y$ and $y=-cx$.
This completes the proof of the lemma.
\end{Proof}

In the PDE dynamics, the region where the origin is stable expands with speed $|c|$, and the front describes how this ``reverse'' quench annihilates the symmetry-breaking state $u>0$ with a delay $x\sim|c|^{1/2}$.

Also, we observe that by translating formula \eqref{e:utilde-sec4} back to the original $x$ and $u$ variables, one finds that to leading order
\[
u(x;c) = \frac{(-c)^{\frac{1}{4}} e^{\frac{x^2}{2c}}}{{\pi^{\frac{1}{4}} 
\left( \mathrm{erf}\left( \frac{x}{\sqrt{-c}} \right) + 1  
\right)^{\frac{1}{2}}}}.
\]
(Here, we recall $y=-cx$, $\xi=\eps y$, $u=-c \tilde{u}$, and $\eps=-1/c^3$.)
Hence, at the origin, one has to leading order
\begin{equation}\label{e:u(0)-cnegative} 
 u(0;c) = \frac{(-c)^{\frac{1}{4}}}{\pi^{\frac{1}{4}}}.
\end{equation}
This is illustrated in the bottom left panel of Fig. 1.

\section{Existence and qualitative properties for 
\texorpdfstring{$c\gg 1$}{Lg}}
\label{s:4}

In this section, we analyze stationary solutions of \eqref{e:aclin} in the limit $c \gg 1$. This is step (iii) in the outline.
To establish the existence, uniqueness, monotonicity, and asymtptotics of these solutions, we scale $y= cx, u=c\tilde{u},$ and $\eps=1/c^3$ and find 
\begin{equation} \label{e:cpos}
 \tilde{u}_{yy}+\tilde{u}_y-\varepsilon y \tilde{u}-\tilde{u}^3=0.
\end{equation}
The construction of the admissible solution now follows closely the construction used in the proof of Theorem 1.1 in \cite{gksv}, with the value of $c$ there set to one. Let $\xi=\eps y$ and write \eqref{e:cpos} as a system,
\begin{equation*}
    \begin{split}
        \tilde{u}_y &= \tilde{v} \\
        \tilde{v}_y &= - \tilde{v} + \xi\tilde{u} +\tilde{u}^3 \\
        \xi_y &= \eps.
    \end{split}
\end{equation*}
In the singular limit ($\eps=0$), the equilibrium $\tilde{u}=\sqrt{-\xi}$, which exists for each $\xi<0$, is a saddle in the $(\tilde{u},\tilde{v})$ plane, so that their union forms a curve of saddle fixed points. The origin $\tilde{u}=0$ is a stable spiral for $\xi< -1/4$, a stable node for $-1/4 < \xi < 0$, and a saddle for $\xi>0$. At $\xi=-1/4$ to leading order, the solution follows a fast heteroclinic orbit connecting the curve $\tilde{u}=\sqrt{-\xi}$ to the origin. 

To obtain persistence of the unstable manifold coming from $\xi = -\infty$ for $0<\varepsilon\ll1$, we once again compactify in $\xi$ by setting $\beta = (-\xi)^{-1/2},\quad \tilde u = \beta^{-1}\bar u, \quad \tilde v = \beta^{-2} \bar v,\quad \frac{d}{dy} = \beta^{-1}\frac{d}{dz}$ for $\xi<0$, obtaining:
\begin{align}
    \bar u_z &= \bar v + \frac{\varepsilon}{2}\beta^3 \bar u,\\
    \bar v_z &=  - \bar u + \bar u^3-\beta \bar v +\varepsilon \beta^3 \bar u,\\
     \beta_z &= \frac{\varepsilon}{2}\beta^4.
\end{align}
(The variable $\beta$ here is different from that used in Section 2.) 
Now, for $\varepsilon = 0$, the curve of saddle equilibria, $\tilde{u} = \sqrt{-\xi}$ for $\xi<0$, corresponds to $\bar u = 1$ for $\beta\geq0$. We denote the union of this curve with the corresponding one-dimensional fast unstable manifolds as $W_0^u$. Then, for $0<\varepsilon\ll1$, Fenichel theory \cite{fenichel} gives that $W_0^u$ persists as the unstable manifold, $W_\varepsilon^u$, of the equilibrium $(\bar u,\bar v,\beta) = (1,0,0)$.   The curve of equilibria persists as a one-dimensional  slow unstable manifold emanating from the equilibrium, and the union of the one-dimensional fast unstable manifolds for $\varepsilon = 0$ persists as a smooth strong unstable foliation of $W_\varepsilon^u$, with base points on the slow unstable manifold. Having established persistence of the unstable manifold for $0<\varepsilon\ll1$, one can convert back into the $y$-variable, and track its fast fibers into a neighborhood of the origin in the $(\tilde u,\tilde v)$ plane. 

Then, for all $0<\eps \ll 1$, Theorem 1.1 of \cite{gksv} establishes that the front is located at $\eps y = -\frac{1}{4} - \Omega_0 (\tfrac{15}{16})^{2/3} \eps^{2/3} + \mathcal{O}(\eps \ln(\eps))$.
Translating this to the original variables, one finds that the front is located at $x \sim -\frac{c^2}{4} - \Omega_0 (\tfrac{15}{16})^{2/3}$, where here $c \gg 1$. This establishes the first statement in Proposition \ref{p:1} (C). The convergence as $\eps\rightarrow0$ of the slow manifolds onto the critical equilibria $\bar u = 1$ for $\beta\geq0$ and $\bar u = 0$ for $\beta <0$ establishes the second statement in Proposition \ref{p:1} (C).
Finally, by desingularizing the origin, one finds that the critical transition happens at a SNIC (saddle-node on an invariant circle) bifurcation induced by a double eigenvalue in the singular circle. 
Also, this construction of the front solution in the transverse intersection of invariant manifolds for $c\gg 1$ shows that it crosses the curve $\sqrt{-x}$ in a unique point $x_0<0$. In sum, we have the following lemma:
\begin{lemma}\label{l:cggp1}
For each $c \gg 1$, there exists a unique admissible solution $u(x;c)$ of \eqref{e:aclin}.
\end{lemma}

\medskip

\section{Existence and monotonicity for all 
{\it c}: continuation analysis}\label{s:5}

Having established existence and uniqueness of admissible solutions for large $|c|$ in Lemmas~\ref{l:cllm1} and \ref{l:cggp1}, we now address steps (iv)--(vi) of the proof of Theorem \ref{t:1} and Proposition \ref{p:1}.

\begin{lemma}[Existence for open sets of $c$]\label{l:ift}
 The set of $c$ such that there exists an admissible solution to  \eqref{e:aclin} is open. Moreover, near each $c_0$ with an admissible solution, there exist a $\delta(c_0)>0$ and a family of admissible solutions that is smooth in $c$ for $|c-c_0|<\delta(c_0)$.
\end{lemma}
\begin{Proof}
Let $u_0$ be the admissible solution at $c=c_0$. We set $u=u_0+w$ for a solution at
a nearby value $c$ and find the equation
\begin{equation}
 F_0(w;c)=\mathcal{L}_{c_0} w + (c-c_0)w_x + (c-c_0)u_{0,x} - 3u_0 w^2-w^3=0.
\end{equation}
We recall that the linear operator is defined by \eqref{e:opL} and its spectrum characterized in Lemma \ref{p:lc}.
We then invert $\mathcal{L}_{c_0}$ and find
 \begin{equation}
 F(w;c)=w + \mathcal{L}_{c_0}^{-1}\left( (c-c_0)w_x + (c-c_0)u_{0,x} - 3u_0 w^2-w^3\right)=0,
 \end{equation}
which we consider as an equation on $H^1$. By definition, $F(0;c_0)=0$, and $\partial_wF(0;c_0)$ is bounded invertible. Moreover, since $\mathcal{D}(\mathcal{L}_c) \subset H^1$ , $\mathcal{L}_c^{-1}$ is bounded from $L^2\to H^1$, and we readily find that $F:H^1(\R)\times\R\to H^1(\R)$ is smooth. The implicit function theorem then gives the desired family of solutions with the prescribed limits. 

It remains to verify that for $|c-c_0|$ sufficiently small, the solution is monotone. Monotonicity in compact intervals follows from continuity in $c$ and the fact that $w_x<0$. 
For monotonicity as $x\to\pm\infty$, first recall that the asymptotic stable (Lem. \ref{l:asyp}) and unstable manifolds (Lem. \ref{l:asym} or \ref{l:asymc}), are smooth in the parameter $c$ and solutions contained in them have asymptotic expansions which also depend smoothly on $c$ and are monotonic in $x$ for each fixed $c$. For admissible solutions considered here, continuity in $c$ in compact $x$-intervals thus implies continuity of the coefficients in their asymptotic expansions as follows. The asymptotic expansions guarantee that for any $\alpha_-\in\R$ and $\alpha_+>0$, there are neighborhoods of infinity, $x<-M$ and $x>M$, respectively, uniform in $|\alpha-\alpha_\pm|<\delta$ for some $\delta>0$ so that $u$ is strictly decreasing. Since also $\alpha_+= 0$ implies $u(x)=0$ for all $x$ and is hence excluded,  we may conclude monotonicity for all $c$ near $c_0$.
\end{Proof}


\begin{lemma}[Monotonicity in $c$]\label{l:m}
  For any  family of admissible solutions $u(x;c)$ of \eqref{e:aclin}, we have $\partial_c u(x;c)<0$ for all $x,c\in\R$.
\end{lemma}
\begin{Proof}
 By taking the derivative of the equation for stationary solutions of \eqref{e:aclin} with respect to $c$, one finds $\mathcal{L}_c \partial_c u = - u_x.$ Hence, 
 \[
  \partial_cu=-\mathcal{L}_c^{-1}\partial_x u.
 \]
Now, since the Sturm-Liouville operator $-\mathcal{L}_c$ is strictly positive, it is also resolvent positive, that is, the associated Green's function $K_c(x,y)$ is positive. In particular, we observe that  $-\mathcal{L}_c$ is self-adjoint in a weighted $L^2$ space with weight $w(x)=e^{c x/2}$ (recall that $f\in L^2(w)$ if $wf \in L^2$), with positive ground-state eigenvalue. Also, the evolution $e^{-\mathcal{L}_c t}$ is positivity preserving for all $c$. Hence,  Theorem XIII.44 of \cite{ReedSimonIV} gives positivity of the resolvent. 

In turn, this establishes that, for all $c$,
\[
 \partial_cu(x;c)=\int_\R K_c(x,y)\partial_xu(y)\rmd y<0,
\]
completing the proof of the lemma.
\end{Proof}

\begin{lemma}[Existence for closed sets of $c$]\label{l:c}
The set of values of $c$ for which there exists an admissible solution is closed.
\end{lemma}
\begin{Proof}
 Given a sequence $\{ c_k \}$ for which $c_k\to c_\infty$ for some $c_\infty\in \mathbb{R}$ and given an admissible solution $u(x;c_k)$ for each $c_k$, we wish to extract a convergent subsequence. Therefore, notice that $u(x;-M)>u(x;c_k)>u(x;M)$ for $M$ sufficiently large. Regularity then immediately gives compactness in the local topology, so that there exists a solution $\bar{u}(x)$ on which the sequence limits, $u(x;c_{k_\ell})\to \bar{u}(x)$, locally uniformly. As a consequence, $\bar{u}(x)$ is a stationary solution of \eqref{e:aclin} at $c=c_\infty$. We need to show that it is in fact admissible.

First, we observe that, by monotonicity in $c$ (recall Lemma \ref{l:m}), there exists an $M>0$ sufficiently large and admissible solutions $u(x;-M)$ and $u(x;M)$ such that
$u(x;-M)>\bar{u}(x)>u(x;M)$ for all $x$. This guarantees that $\bar{u}(x)$ has the asymptotic behavior required of an admissible solution.

To show monotonicity in $x$, namely $\partial_x\bar{u}(x)<0$, we argue by contradiction. That is, we assume there exists an $x_0$ at which $\partial_x\bar{u}(x_0)=0$; this suffices due to the asymptotics as $x \to \pm \infty$. Since the solutions $u(x;c_{k_\ell})$ in the approximating sequence are monotone, it must be that $\partial_{xx}\bar{u}(x_0)=0$ and $\partial_{xxx}\bar{u}(x_0)> 0$ (where the latter is derived by taking the derivative of equation \eqref{e:aclin}). Substituting the two equalities into the equation for stationary solutions of \eqref{e:aclin}, we find that $x_0\bar{u}(x_0)+\bar{u}^3(x_0)=0$, which implies $x_0<0$. Furthermore, the inequality implies that $\partial_{xx} \bar{u}>0$ for $x> x_0$ and close to $x_0$. 

Then, recalling the asymptotics as $x \to \infty$, it must be that $\bar{u}$ has a maximum on the interval $x>x_0$ and that the value of $\bar{u}$ at the maximum is such that $\bar{x}\bar{u}+\bar{u}^3>0$. However, this gives a contradiction, by the maximum principle. Hence, there cannot be such a point $x_0$, and we have shown that $\partial_x \bar{u}(x) < 0 $ for all $x$.
This completes the demonstration that $\bar{u}$ is admissible.
\end{Proof}

We are now ready to complete the proof of our main result.

\begin{Proof}[ of Theorem \ref{t:1}]
The set of $c\subset \R$ such that there exists an admissible solution is open and closed by Lemmas \ref{l:ift} and \ref{l:c}. It is also nonempty since it includes values of $|c|$ sufficiently large, by Lemmas \ref{l:cllm1} and \ref{l:cggp1}, and since it includes the Hastings-McLeod solution for $c=0$. Therefore, admissible solutions exist for all $c\in\R$, because if solutions were to cease being admissible for some finite value of $c$ then the set could not be both open and closed.

It remains to show uniqueness. Lemmas \ref{l:cllm1} and \ref{l:cggp1} establish uniqueness in the limits $|c|\gg 1$. Now, if there is more than one admissible solution for some finite value of $c$, then Lemma \ref{l:ift} guarantees that the set of such values is open. Moreover, Lemma \ref{l:c} would show that each of the solution branches can be continued to a solution branch on all of $\R$. However, since solution branches are isolated by Lemma \ref{l:ift}, this would imply the existence of multiple branches for large $|c|$, which is a contradiction. Hence, for each $c$ the solution is unique.
\end{Proof}

The proofs above establish Theorem \ref{t:1}, as well as properties (A)--(D) of Proposition \ref{p:1}. 
It remains to establish property (E).
This is done in the following lemma.
We observe that this unique intersection property for all $c\ge 0$ is the natural extension of the known result for $c=0$, where the Hastings-McLeod solution $u_{\rm HM}=u_*(x;0)$ is known (see \cite{hastings80}) to intersect the curve $u=\sqrt{-x}$ in a unique point on $(-\infty,0)$. Moreover, this point is the unique inflection point of the Hastings-McLeod solution. Hence, it is useful to establish the same unique intersection point property for the admissible solutions of the Painlev\'e-II equation with drift term, {\it i.e.,} for the fronts of \eqref{e:aclin} for any $c\ge 0$.

\begin{lemma}\label{l:cross}
 For each $c\ge 0$, the admissible solution $u$ possesses precisely one value $x_0<0$ such that $x_0u(x_0)+u(x_0)^3=0$.
\end{lemma}
\begin{Proof}
 The result is known for $c=0$, see Theorem 1 in \cite{hastings80}, where it is shown that the Hastings-McLeod solution $u_{\rm HM} = u_*(x;0)$ which is the unique admissible solution for $c=0$ has a unique inflection point $x_0<0$ precisely where it crosses the curve $\sqrt{-x}$. Furthermore, this crossing is transverse, with $\frac{d}{dx}\sqrt{-x}|_{x_0} = \frac{-1}{2\sqrt{-x_0}} < \partial_x u_*(x_0;0) < 0$.

Next, for each $c \gg 1$, inspection of the solution shows that there is exactly one point $x_0$ where $x_0 u+u^3=0$. Furthermore, for these $c$, the admissible solution $u_*(x;c)$ intersects the curve $\sqrt{-x}$ at $x_0$ transversely.

Hence, to prove the lemma, we need to show that the result holds for any finite $c>0$. We do this using a proof by contradiction. We start by observing that, for there to be $c$ value(s) for which the admissible solution has more than one point of intersection, there must be some value of $c$ for which the admissible solution $u_*(x;c)$ has a point of tangency with the curve $\sqrt{-x}$. Hence, we set $u_*(x;c) = \sqrt{-x} + v(x)$ and assume that there exists a $c_1>0$ and a point $x_1<0$ such that, for $u_*(x;c_1)$, one has $v(x_1)=0$ and $v'(x_1)=0$. Next, we calculate derivatives and find
 \begin{equation} \label{e:lemma64}
  0=v''+cv'-x v-v^3-3(-x)v-3(-x)^{1/2}v^2-\frac{1}{4}(-x)^{-3/2} - \frac{c}{2} (-x)^{-1/2}.
 \end{equation}
 Hence, one sees that for $c=c_1$,
 \begin{equation*}
 v''(x_1) = \frac{1}{4(-x_1)^{3/2}} + \frac{c_1}{2\sqrt{-x_1}} > 0.  
 \end{equation*}
From this, it follows that
\begin{equation*}
\partial_{x}^2 u_*(x_1;c_1) = \frac{d^2}{dx^2}\sqrt{-x}\Big\vert _{x_1} + v''(x_1) 
= \frac{c_1}{2\sqrt{-x_1}}>0.
\end{equation*}
However, this is a contradiction, since $\partial_x^2 u_*(x_1;c_1)$ must be non-positive at any such point $x_1$,
{\it i.e.,} $v''(x_1)$ must be less than or equal to zero.
In fact, the above calculation shows that there no value of $x<0$ at which $v$ can have a double root ($v,v'=0$) where $v''\leq 0$. 

Therefore, we have shown that the number of roots is constant unless the sign of the asymptotics of $v$ at $\infty$ changes. However, from the $x \to \infty$ asymptotics of admissible solutions with $c>0$ one sees that this sign remains the same. Hence, there is a unique intersection, and the proof is complete.

\end{Proof}

\section{Results from numerical simulations}\label{s:num}

In this section, we present results of numerical simulations that go beyond the basic phenomena shown in Figure~\ref{f:num}, to illustrate and extend the conclusions of Theorem~\ref{t:1} and Proposition~\ref{p:1}.

Front solutions of the full PDE \eqref{e:actanh} with the hyperbolic tangent ramp function  were computed using AUTO07p \cite{doedel2007auto}, while admissible solutions of the PDE \eqref{e:aclin} with the linear ramp were computed using natural parameter continuation in $\tilde c$ with fourth order finite differences, centered for $\partial_{\tilde x}^2$ and with either up-winding or down-winding for $\partial_{\tilde x}$ depending on whether $\tilde c>0$ or $\tilde c<0$, respectively. In the latter case, the discretization size was $d\tilde x = 0.01$, while the domain-length was $\tilde L = 300.$ 

\begin{figure}[h]
 \centering
 \hspace{-0.7in}
 \includegraphics[width=0.73\textwidth]{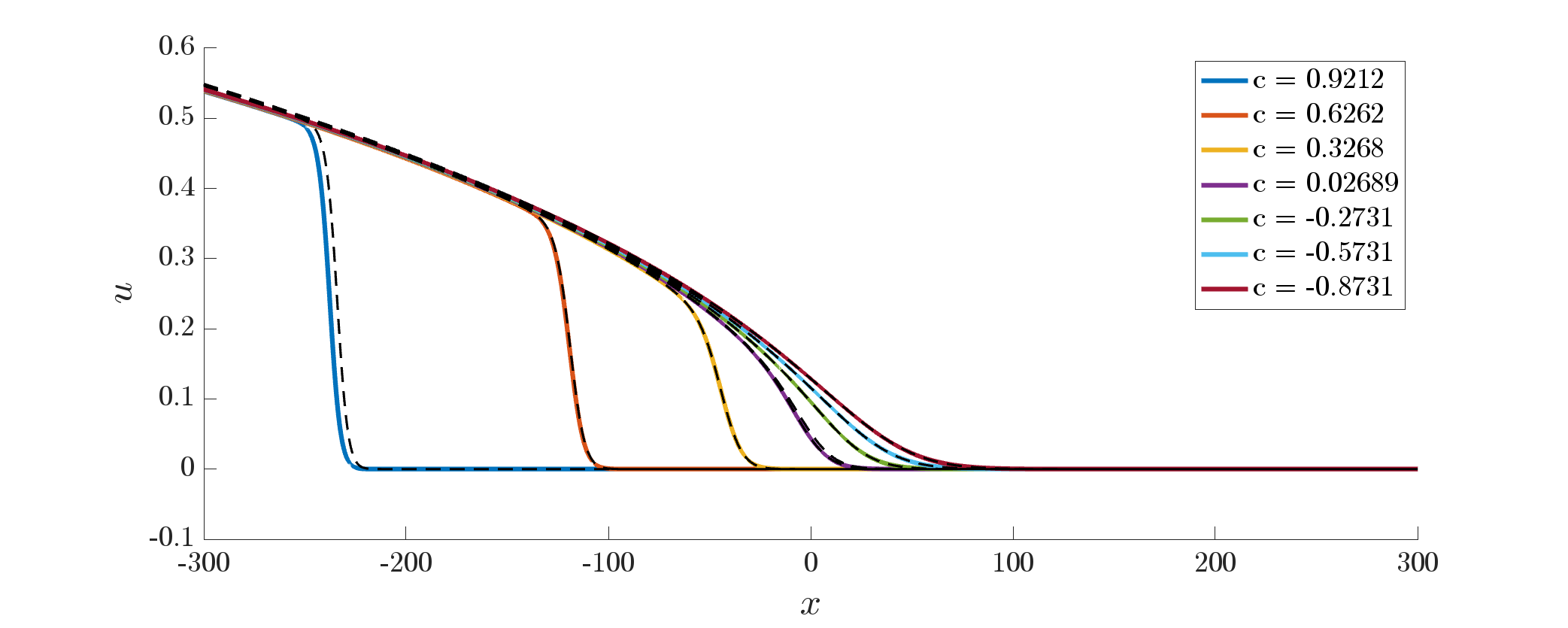}
 \hspace{-0.5in}
\includegraphics[width=0.35\textwidth]{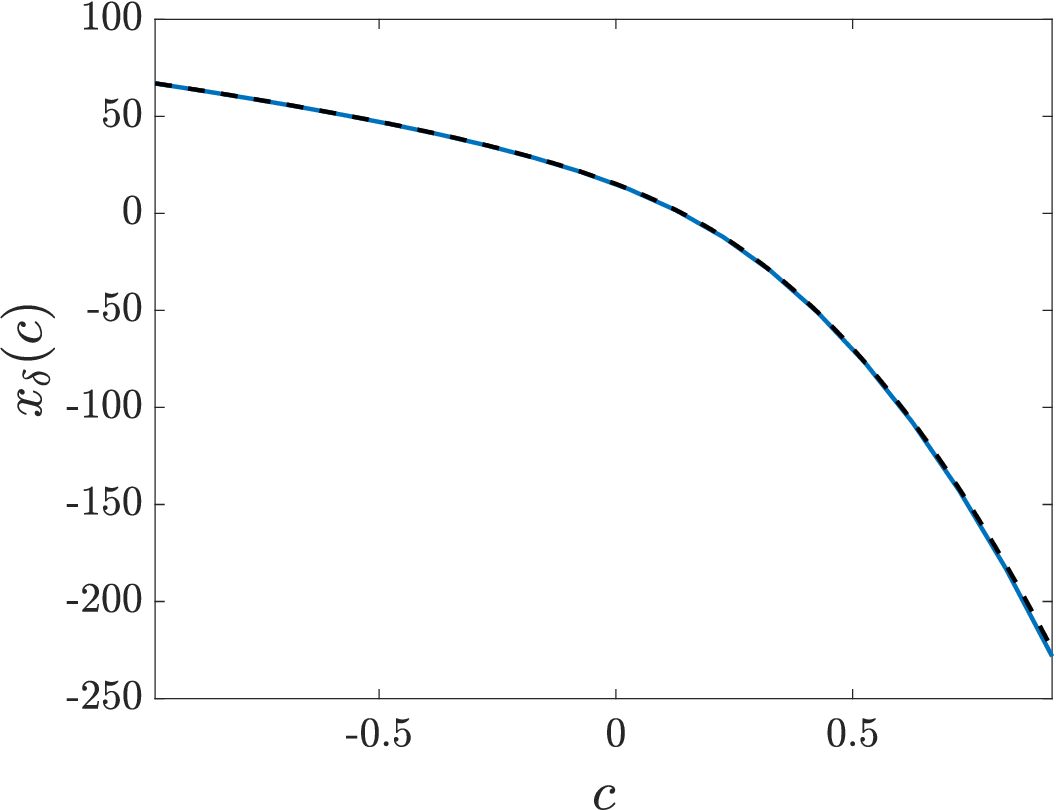}
  \caption{Left: Comparison of front solutions of \eqref{e:actanh}, $u(x;c)$ (solid colored lines) with scaled inner solutions given by $\varepsilon^{1/3}\tilde u(\varepsilon^{-1/3} x)$ (dashed black lines) for a range of un-scaled $c$ values (see legend).
  Right: Comparison of the front interface location $x_\delta(c)$ (blue solid) in \eqref{e:actanh} with the scaled front interface, $\tilde x_\delta(\varepsilon^{-1/3} c)$ (black dashed), for \eqref{e:aclin}. Here, $\delta = 0.1$ and $\varepsilon = 0.001$ throughout.}
  \label{f:tanh}
\end{figure}

The admissible solutions studied in this article accurately describe the ``inner" dynamics of the Allen-Cahn fronts when the ramp is the hyperbolic tangent function used here or a more general, step-like function.
We recall that the variables $u,c, x$ denote the original unscaled variables in \eqref{e:actanh} and that the variables $\tilde u, \tilde c, \tilde x$ represent the scaled variables given in \eqref{e:scale} for the Painlev\'{e}-II with drift equation \eqref{e:aclin}. For $|c| = \mathcal{O}(\varepsilon^{1/3})$, we find that appropriately-scaled admissible solutions of \eqref{e:aclin} provide accurate inner solutions in $|x|\lesssim \varepsilon^{-1/3}$ for traveling-waves in the original slowly-ramped Allen-Cahn equation \eqref{e:actanh} with hyperbolic tangent heterogeneity.  
Equilibrium front solutions are depicted for both  equations in Figure \ref{f:tanh} (left) for a range of $c$-values.  

Next, in Figure \ref{f:tanh} (right), we compare the measured front position $\tilde x_\delta(c)$ for \eqref{e:aclin}, defined by $\tilde u(\tilde x_\delta(\tilde c);\tilde c) = \delta$, for some fixed $\delta>0$ small, with the corresponding front position $x_\delta(c)$ for \eqref{e:actanh}, defined by $u(x_\delta(c);c) = \varepsilon^{1/3} \delta$. We find excellent agreement between $x_\delta(c)$ and  $\varepsilon^{-1/3} \tilde x_\delta ( \varepsilon^{-1/3} c)$. We only find divergence in the two solutions for $\mathcal{O}(1)$ $c$-values where the front interface $x_\delta(c)$ is outside the region where $-\tanh(\varepsilon x) \sim -\varepsilon x$. 

We reiterate that these numerics confirm that the scaled solutions obtained in this work accurately describe the inner solutions of the front solutions $u(x;c)$ in \eqref{e:actanh}.

\section{Conclusions}\label{s:concl}

\subsection{Summary of the main results}

In  this article, we studied the PDE \eqref{e:aclin} that arises as a model of spatial slow passage through a pitchfork bifurcation. 
We established the existence, uniqueness, monotonicity, quantitative asymptotics, and qualitative properties of a class of front solutions (which we labeled as ``admissible" solutions)  of \eqref{e:aclin} for all values of the parameter $c \in {\mathbb R}$. (Here, we recall that the tilde has been dropped in Sections 2--6, so that the $c$ in \eqref{e:aclin} is $\tilde{c}$, where the speed $c$ of the quenching front in \eqref{e:actanh} has been scaled as $c = \varepsilon^{1/3}\tilde{c}$.) The admissible solutions connect the state $u=\sqrt{-x}$ as $x \to -\infty$ to the state $u=0$ as $x \to \infty$. 
The spatial decay of these fronts toward the asymptotic states as $x \to \pm \infty$ is given by Lemmas~\ref{l:asyp}-\ref{l:asymc}.
We showed that, for each $c \in {\mathbb R}$, the fronts are monotone decreasing in $x$ (Theorem \ref{t:1}), and that at each point in space, they are monotone decreasing in $c$ (Lemma~\ref{l:m}). In addition, precise asymptotics were given in the limits $c \to -\infty$ (Lemma~\ref{l:cllm1}) and $c \to \infty$ (Lemma~\ref{l:cggp1}).

As $c$ varies over all real numbers, it was shown that the unique monotone decreasing fronts make a transition from exhibiting diffusive spill-over for $c\le 0$ to exhibiting a significant delay in the onset of the instability post the pitchfork bifurcation, and hence that the front is located at a substantial distance away from where the instability first occurs. The asymptotics of the front location ({\it i.e.,} of this delayed onset) for large $c>0$ are given in Property (C) of Proposition~\ref{p:1}.

A variety of methods from classical asymptotic analysis, dynamical systems, invariant manifold theory, spectral theory of Schr\"odinger operators, and functional analysis were used to establish the various lemmas and the main theorem.
We highlight the demonstrations
that the set of parameter values $c$ for which admissible solutions exist is both open (Lemma~\ref{l:ift}) and closed (Lemma~\ref{l:c}). These properties were established by showing that
key operators obtained from linearizing about admissible solutions are bounded invertible (Lemma \ref{p:lc})
and by using other functional analytic techniques. 
Therefore, because there exist unique admissible solutions for each $\vert c\vert \gg 1$ and for $c=0$, it followed that admissible solutions must exist for all $c$.

We also presented numerical simulations of the PDEs to confirm the theory and to illustrate the quantitative asymptotics and properties of the fronts, especially the transition from diffusive spill-over to delayed fronts. See Figures~\ref{f:num} and \ref{f:tanh}. 

The PDE \eqref{e:aclin} arises as the inner problem for the PDE \eqref{e:actanh} and for other PDEs modeling quenching with slow spatial ramps. It zooms in on the critical parameter regime $c=\mathcal{O}(\varepsilon^{1/3})$ in \eqref{e:actanh} about the pitchfork bifurcation. In this respect, the analysis in this article complements our recent work \cite{gksv}.
There, the diffusive spill-over of fronts was established for $c=0$ in \eqref{e:actanh} (see Theorem 1.1 in \cite{gksv}),
and the delayed onset of fronts was proven for $\mathcal{O}(1)$ values of $c>0$ (see Theorem 1.2 in \cite{gksv}). The scaling analyzed here $(c=\varepsilon^{1/3} \tilde{c}$) is precisely that in the critical transition regime between these two earlier cases. Hence, it is useful for understanding the transition in this and related PDEs. Moreover, we have analyzed system 
\eqref{e:aclin} for all $c \in {\mathbb R}$, so that we not only cover fully the inner domain, {\it i.e.,} we link up the case $c=0$ in \eqref{e:actanh} in the limit as $\tilde{c} \to 0^+$\ to the case $c>0$ and $\mathcal{O}(1)$ in the limit as $\tilde{c} \to \infty$, but we also include the regime $c<0$ in \eqref{e:aclin} to present a full dynamical systems  and unfolding analysis as $c$ passes through zero. 

Besides being of interest for quenching problems with slow spatial ramps, the admissible solutions studied here are also of interest for extending the Hastings-McLeod solution of the Painlev\'e-II equation (which is the unique admissible solution for $c=0$) to Painlev\'e-II equations with a drift term, which arise naturally here as the equation for stationary solutions of \eqref{e:aclin}.

Finally, our results may be thought of as providing a PDE analogue to the temporal passage through a pitchfork bifurcation in ODEs, where long, $\caO(1)$ delays are expected when passing from a subcritical to a supercritical parameter regime, but only small delays occur in the reverse transition, recall \cite{Krupa_2001,maree96}.

\subsection{Discussion and future directions}

In this section, we discuss a number of avenues for future investigation.

Our methods contain a blend of geometric tools to describe asymptotics at spatial infinity, and a more traditional functional analytic continuation argument based on the absence of saddle-nodes and a priori estimates. It would be interesting to understand the existence problem as a shooting problem in $\R^3$. Alternatively, it seems plausible that a more direct argument based on a Leray-Schauder degree would give existence for any fixed $c$. More PDE oriented methods based on sub- and super solutions may also give additional insight into the shape of solutions, and possibly generalize well to more geometric questions in  higher space-dimensions.

Beyond the monotone ramps studied in this article, it would be interesting to study non-monotone ramps and the possibility of creating non-monotone structures in the wake. We emphasize that monotonicity enters crucially at several points in our proofs, in particular when ruling out saddle-node bifurcations through establishing negativity of the linearization. 

The pitchfork bifurcation studied in this article is just one of many prototypical examples of systems with spatially varying ramps that have small gradients. One would clearly wish to have similarly detailed descriptions of slow spatial passages through other elementary bifurcations, such as saddle-node, transcritical, subcritical pitchfork, and Hopf bifurcations. Some of the relevant partial work is summarized in \cite{gksv}. Other interesting examples include pattern-forming bifurcations, such as, for instance,  a Swift-Hohenberg equation with a spatial ramp passing through a Turing, Eckhaus, or zigzag instability. It would also be of interest how pattern-forming phenomena arising from slow spatial ramps compare with those explored in temporally dynamic bifurcations; see for example \cite{asch2023slow,jelbart2023formal,tsubota2023bifurcation}.

The results here can be thought of as describing ramps in multiple space dimensions, when the ramp itself simply does not depend on a second or third spatial direction. Fronts found here would then also be stable against perturbations in this transverse direction. On the other hand, one could first ask if more complex states can arise in the wake of such a one-dimensional parameter ramp when observed in higher ambient space dimensions. For step-like parameter ramps, rather than slow ramps, these questions were studied in \cite{monteiro1,monteiro3,monteiro2}, uncovering some peculiar constraints. Understanding more generally the impact of the geometry of the quenching ramp on the possibility of complex patterns in the wake of the quenching process appears to be a wide open question.

  \section*{Acknowledgments}
  The authors were partially supported by the National Science Foundation through grants NSF-DMS-2006887, DMS-2307650 (RG), NSF-DMS-1616064 (TK), and NSF DMS-1907391 and DMS-2205663 (AS). TK and AS would like to thank the Banff International Research Station for its hospitality during the workshop ``Topics in Multiple Time Scale Dynamics" in November 28 -- December 2, 2022, as well as the organizers Maximilian Engel, Hildeberto Jardon-Kojakhmetov, Bj\"orn Sandstede, and Cinzia Soresina.

\bibliographystyle{abbrv}
\bibliography{p2}

\end{document}